\documentclass[12pt,draft]{amsart}

\usepackage{amsrefs}
\usepackage{amsmath}
\usepackage{amssymb}
\usepackage{color}
\usepackage{mathrsfs}
\usepackage[all]{xy}
\usepackage{tikz}
\usetikzlibrary{decorations.markings}

\def\altdb{\vadjust{\vbox to 0pt{\vss\hbox{\kern \hsize
\quad{\dbend}}\kern\baselineskip\kern-10pt}}}

\newcommand\sset[1]{\{#1\}}
\newcommand\set[1]{\{\,#1\,\}}
\newcommand\Iso{\operatorname{Iso}}

\newcommand\go{G^{(0)}}
\newcommand\restr[1]{|_{#1}}
\newcommand\cs{\ensuremath{C^{*}}}

\newcommand\cc{C_{c}}
\DeclareMathOperator{\supp}{supp}
\DeclareMathOperator{\ev}{ev}
\DeclareMathOperator{\lsp}{span}
\DeclareMathOperator{\clsp}{\overline{\lsp}}
\newcommand{\inv}{^{-1}}

\newtheorem{thm}{Theorem}[section]

\newtheorem{lemma}[thm]{Lemma}
\newtheorem{prop}[thm]{Proposition}
\newtheorem{cor}[thm]{Corollary}

\theoremstyle{definition}

\theoremstyle{remark}
\newtheorem{remark}[thm]{Remark}

\newtheorem{example}[thm]{Example}

\newtheorem{mycomment}{Comment}
{\end{mycomment}\endgroup}

\makeatletter
\def\labelenumi{\textnormal{(\@alph\c@enumi)}}
\def\theenumi{\@alph \c@enumi}
\def\labelenumii{\textnormal{(\@roman\c@enumii)}}
\def\theenumii{\@roman \c@enumii}
\newcount\charno
\def\alphapart#1{\charno=96
\advance\charno by#1\char\charno}

\makeatother
\numberwithin{equation}{section}
\evensidemargin=\oddsidemargin
\newcount\hours
\newcount\minutes       
\def\timeofday{
\hours=\time
\minutes=\hours
\divide\hours by60
\multiply\hours by60
\advance\minutes by-\hours
\divide\hours by60
\ifnum\hours>9\else0\fi\the\hours:\ifnum\minutes>9\else
0\fi\the\minutes}
\def\predate{\date{\color{red}\bfseries \the\day\ \ifcase\month\or
  January\or February\or March\or April\or May\or June\or July\or
        August\or September\or October\or November\or
           December\fi\ \the\year\ --- \timeofday}}

\let\phi\varphi
\def\intiso(#1){\Iso(#1)^{\circ}}

\newcommand\Bb{\mathscr{B}}

\newcommand\N{\mathbf{N}}

\newcommand\T{\mathbf{T}}
\newcommand\Z{\mathbf{Z}}

\DeclareMathOperator{\Aut}{Aut}

\newcommand{\Hh}{\mathcal{H}}

\allowdisplaybreaks[1]
\renewcommand\MR[1]{\relax}

\usepackage[normalem]{ulem} 
\usepackage{color}
\definecolor{Dgreen}{cmyk}{0.93,0.33,0.92,0.25} 

\begin{document}
\title[Cartan subalgebras of groupoid \cs-algebras]%
{\boldmath Cartan subalgebras in \cs-algebras of Hausdorff \'etale groupoids}

\author[J.H. Brown]{Jonathan H. Brown}
\address[J.H. Brown]{
Department of Mathematics\\
University of Dayton\\
300 College Park Dayton\\
OH 45469-2316 USA} \email{jonathan.henry.brown@gmail.com}

\author[G. Nagy]{Gabriel Nagy}
\address[G. Nagy and S. Reznikoff]{
Department of Mathematics\\
Kansas State University\\
138 Cardwell Hall\\
Manhattan, KS\\
U.S.A.} \email{nagy@math.ksu.edu}

\author[S. Reznikoff]{Sarah Reznikoff}
\email{sarahrez@math.ksu.edu}

\author[A. Sims]{Aidan Sims}
\address[A. Sims]{School of Mathematics and Applied Statistics\\
University of Wollongong \\
NSW 2522\\
Australia} \email{asims@uow.edu.au}

\author[D.P. Williams]{Dana P. Williams}
\address[D.P. Williams]{Department of Mathematics \\ Dartmouth College
  \\ Hanover, NH 03755-3551}

\email{dana.williams@Dartmouth.edu}



\thanks{We thank Astrid an Huef and Lisa Orloff Clark for helpful
  conversations.  We are grateful to Alex Kumjian for pointing out an
  error in an earlier version of the manuscript. This research was
  supported by the Australian Research Council, the Edward Shapiro
  fund at Dartmouth College, the Simons Foundation, and the National
  Science Foundation grant number DMS-1201564}

\date{\today}

\begin{abstract}
  The reduced $C^*$-algebra of the interior of the isotropy in any
  Hausdorff \'etale groupoid $G$ embeds as a $C^*$-subalgebra $M$ of
  the reduced $C^*$-algebra of $G$. We prove that the set of pure
  states of $M$ with unique extension is dense, and deduce that any
  representation of the reduced $C^*$-algebra of $G$ that is injective
  on $M$ is faithful. We prove that there is a conditional expectation
  from the reduced $C^*$-algebra of $G$ onto $M$ if and only if the
  interior of the isotropy in $G$ is closed. Using this, we prove that
  when the interior of the isotropy is abelian and closed, $M$ is a
  Cartan subalgebra. We prove that for a large class of groupoids $G$
  with abelian isotropy---including all Deaconu--Renault groupoids
  associated to discrete abelian groups---$M$ is a maximal abelian
  subalgebra. In the specific case of $k$-graph groupoids, we deduce
  that $M$ is always maximal abelian, but show by example that it is
  not always Cartan.
\end{abstract}

\maketitle


\section{Introduction}
\label{sec:introduction}

A key tool in the study of graph $C^*$-algebras and their analogues is
the Cuntz--Krieger uniqueness theorem \cite{cunkri:im80,
  kumpasraeren:jfa97}. This result says that if all the cycles in a
graph $E$ have an entrance, then any representation of the associated
$C^*$-algebra which is nonzero on all of the generating projections
associated to vertices of the graph is faithful. There are
numerous ways to prove this theorem.  But the key to the argument in
each case is showing that any element of the graph $C^*$-algebra can
be compressed to an element close to its canonical abelian subalgebra,
and this process is faithful on positive elements.  This compression
property is reminiscent of Anderson's study \cite{and:tams79} of the
state-extension property for inclusions of $C^*$-algebras.

In the case of graph algebras, the condition on a graph that every
cycle has an entrance is equivalent to the condition that the
associated groupoid is topologically principal. (It is worth pointing
out that this is in fact how the Cuntz--Krieger uniqueness theorem was
originally proved \cite[Theorem~3.7]{kumpasrae:pjm98}.) It follows
from Renault's work in \cite{ren:imsb08} that this in turn is
equivalent to the condition that there is a dense set of units of the
groupoid $G$ for which the associated pure state of $C(\go)$ has
unique extension to $C^*_r(G)$. It also follows from Renault's work
that when $G$ is topologically principal $C_0(\go)$ is a maximal
abelian subalgebra---indeed, a Cartan subalgebra---of $C^*(G)$.

This analysis fails if $E$ contains cycles with no
entrance. Szyma\'nski showed in \cite{szy:ijm02} that to verify
faithfulness of a representation $\pi$ of $C^*(E)$, in addition to
checking that each $\pi(p_v)$ is nonzero, one must check that
$\pi(s_\mu)$ has full spectrum for every cycle $\mu$ with no
entrance. The second and third authors systematised and generalised
Szyma\'nski's analysis in \cite{nagrez:jlms12, nagrez:pams14}. They
introduced the notion of a pseudo-diagonal $M$ of a $C^*$-algebra $A$
and showed that representations of $A$ that are faithful on $M$ are
automatically faithful on $A$. They then showed that the subalgebra of
a graph $C^*$-algebra generated by the usual abelian subalgebra and
the elements $\{s_\mu : \mu\text{ is a cycle with no entrance}\}$ is a
pseudo-diagonal, recovering Szyma\'nski's result.

The first three authors considered the extension of this analysis to
$C^*$-algebras of higher-rank graphs in \cite{bnr:jfa14}. They
considered an abelian subalgebra $M$ of the $k$-graph algebra
$C^*(\Lambda)$ spanned by partial unitaries of the form
$s_\mu s^*_\nu$, and identified it as the completion in the associated
groupoid $C^*$-algebra of the functions supported on the interior of
its isotropy. By careful analysis of the set of states of $M$ with
unique extension to $C^*(\Lambda)$, they proved that every
representation of $C^*(\Lambda)$ that is injective on $M$ is faithful,
without proving that $M$ was either maximal abelian or the range of a
faithful conditional expectation from $C^*(\Lambda)$. They left open
the natural question as to whether $M$ is in fact a pseudo-diagonal in
the sense of \cite{nagrez:pams14}.

Here we answer a more general question about a canonical subalgebra of
the reduced $C^*$-algebra of a Hausdorff \'etale groupoid $G$. The
reduced $C^*$-algebra of the interior $\intiso(G)$ of the isotropy in
$G$ embeds as a subalgebra $M_r$ of $C^*_r(G)$. We show that the set
of pure states of $M_r$ with unique extension to $C_r^*(G)$ is dense
in the set of all pure states of $M_r$. We conclude from this that any
representation of $C^*_r(G)$ that is injective on $M_r$ is
faithful. This generalises the result from \cite{bnr:jfa14} discussed
in the preceding paragraph. Along the way we show that commutativity
of the subalgebra can be dropped from the hypotheses of the abstract
uniqueness theorem of \cite{bnr:jfa14}.

We then turn our attention to deciding when $M_r$ is a Cartan
subalgebra in the sense of \cite{ren:imsb08}, and when it is a
pseudodiagonal in the sense of \cite{nagrez:pams14}.  It turns out
that the two are equivalent and hold precisely when there is a
conditional expectation $\Psi : C^*_r(G) \to M_r$, and $M_r$ is
maximal abelian. We prove that if $\intiso(G)$ is closed in $G$, then
the map $f \mapsto f|_{\intiso(G)}$ from $C_c(G)$ to $C_c(\intiso(G))$
extends to a faithful conditional expectation
$\Psi : C^*_r(G) \to M_r$; and conversely, if $\intiso(G)$ is not
closed in $G$ then there does not exist a faithful conditional
expectation on $C^*_r(G)$ with range $M_r$.

To address the question of when $M_r$ is maximal abelian, we restrict
our attention to groupoids $G$ such that $\intiso(G)$ is abelian (that
is, a bundle of abelian groups).  Renault shows in
\cite{ren:groupoid}*{Proposition~II.4.2(i)} that the identity map on
$C_c(G)$ extends to a continuous injection $j : C^*_r(G) \to
C_0(G)$.
We show that $M_r$ is maximal abelian if and only if it is equal to
the set of elements $a \in C^*_r(G)$ such that
$j(a) \in C_0(\intiso(G))$. We then establish two sufficient
conditions under which $M_r$ is maximal abelian: (a) that the interior
of the isotropy is closed; or (b) that there is a continuous 1-cocycle
from $G$ into an abelian group $H$ that is injective on each fibre of
the isotropy of $G$. It follows from this and our previous result that
whenever $\intiso(G)$ is abelian and closed, $M_r$ is a Cartan
subalgebra and a pseudodiagonal. It also follows that if $G$ is a
Deaconu--Renault groupoid associated to an action of a discrete
abelian group, and if the interior of the isotropy is not closed in
$G$, then $M_r \subseteq C^*_r(G)$ is a maximal abelian subalgebra of
$C^*_r(G)$ but is not Cartan or a pseudodiagonal because it is not the
range of a conditional expectation.

Specialising to $k$-graphs we are able to answer the questions left
open in \cite{nagrez:pams14}: given a $k$-graph $\Lambda$, the
subalgebra $M$ of $C^*(\Lambda)$ described in \cite{nagrez:pams14} is
always a maximal abelian subalgebra, but is a Cartan subalgebra only
under the additional hypothesis that the interior of the isotropy in
the $k$-graph groupoid is closed. We prove by example that the latter
is not automatic. We also tie up a loose end by proving that if $D$ is
a Cartan subalgebra of any $C^*$-algebra $A$, then $D$ is a
pseudodiagonal as well.

\smallskip

The paper is organised as follows. After a short preliminaries section
to establish notation, we break our analysis up into two sections. In
Section~\ref{sec:abstractuniqueness} we prove our main uniqueness
result about the reduced $C^*$-algebra of a Hausdorff \'etale groupoid
in terms of the subalgebra $M_r$ corresponding to the interior of its
isotropy. The results in this section do not require $\intiso(G)$ to
be abelian or closed. We have tried to be explicit about which parts
of our results apply to full $C^*$-algebras, and in particular what
additional consequences follow from amenability of $G$ or of
$\intiso(G)$.

Section~\ref{sec:abelian isotropy} then deals with the questions of
when there is a conditional expectation of $C^*_r(G)$ onto $M_r$, and
when $M_r$ is maximal abelian. We prove that $C^*_r(G)$ admits an
expectation onto $M_r$ if and only if $\intiso(G)$ is closed in
Proposition~\ref{prp:M FCE}. We then restrict to the special case
where $\intiso(G)$ is abelian and hence also amenable by results of
\cite{ren:xx13}. We establish our sufficient conditions for $M_r$ to
be maximal abelian in Theorem~\ref{thm:masa}. We also discuss the
consequences of our results for higher-rank-graph $C^*$-algebras, and
provide an example of a $2$-graph for which the interior of the
isotropy in the associated groupoid is not closed. We finish the
section by proving that every Cartan subalgebra is a pseudo-diagonal.

\section{Preliminaries}
\label{sec:preliminaries}

Throughout this paper, $G$ will denote a locally compact
second-countable Hausdorff groupoid which is \'etale in the sense that
$r,s : G \to \go$ are local homeomorphisms.  For subsets
$A, B \subset G$, we write
\begin{equation*}
  AB:=\sset{\alpha\beta\in G:(\alpha,\beta)\in (A\times B)\cap G^{(2)}}.
\end{equation*}
We use the standard groupoid conventions that $G^u = r^{-1}(u)$,
$G_u = s^{-1}(u)$, and $G^u_u = G^u \cap G_u$ for $u \in \go$. For
$K\subset \go$, the restriction of $G$ to $K$ is the subgroupoid
$G\restr K=\sset{\gamma\in G:r(\gamma),s(\gamma)\in K}$. We will be
particularly interested in the \emph{isotropy subgroupoid}
\begin{equation*}
  \Iso(G)=\set{\gamma\in G:r(\gamma)=s(\gamma)}=\bigcup_{u\in\go} G^u_u.
\end{equation*}
Note that $\Iso(G)$ is closed in $G$ as well as a group bundle over
$\go$.

The $I$-norm on $\cc(G)$ is defined by
\[
\|f\|_I = \sup_{u \in \go} \max\big\{\sum_{\gamma \in G_u}
|f(\gamma)|, \sum_{\gamma \in G^u} |f(\gamma)|\big\}.
\]
The groupoid $C^*$-algebra $\cs(G)$ is the completion of $\cc(G)$ in
the norm $\|a\| = \sup\{\pi(a) : \pi\text{ is an $I$-norm bounded
  $*$-representation}\}$. For $u \in \go$ there is a representation
$L^u : \cs(G) \to B(\ell^2(G_u))$ given by
$L^u(f)\delta_\gamma = \sum_{s(\alpha) = r(\gamma)}
f(\alpha)\delta_{\alpha\gamma}$.
This is called the (left-)regular representation associated to
$u$. The reduced groupoid $C^*$-algebra $\cs_r(G)$ is the image of
$\cs(G)$ under $\bigoplus_{u \in \go} L^u$.

A \emph{bisection} in $G$, also known as a $G$-set, is a set
$U \subset G$ such that $r,s$ restrict to homeomorphisms on $U$. An
important feature of \'etale groupoids is that they have plenty of
open bisections: Proposition~3.5 of \cite{exe:bbms08} together with
local compactness implies that the topology on an \'etale groupoid has
a basis consisting of precompact open bisections.

Because $G$ is \'etale, there is a homomorphism
$C_0(\go) \hookrightarrow \cs(G)$ implemented on $\cc(G)$ by extension
of functions by $0$. We regard $C_0(\go)$ as a $*$-subalgebra of
$\cs(G)$.  Since $\go$ is closed, $f|_{\go}\in C_c(\go)$ for all
$f\in C_c(G)$.  This map extends to a faithful conditional expectation
$\Phi_r: C^*_r(G)\to C_0(\go)$ and a (not necessarily faithful)
conditional expectation $\Phi: C^*(G)\to C_0(\go)$ (this is proved for
principal groupoids in the final sentence of the proof of
\cite{ren:groupoid}*{Proposition~II.4.8}, and the same proof applies
for non-principal groupoids).

We write $\intiso(G)$ for the interior of $\Iso(G)$ in $G$.  Since $G$
is \'etale, $\go \subset \intiso(G)$ and $\intiso(G)$ is an open
\'etale subgroupoid of $G$. We will need the following consequence of
\cite[Proposition~2.5]{simwil:xx15}.

\begin{lemma}[{\cite[Proposition~2.5(b)~and~(c)]{simwil:xx15}}]\label{lem:normal}
  Suppose that $G$ is a second-countable locally compact Hausdorff
  \'etale group\-oid. For each $\gamma\in G$, the map
  $\alpha\mapsto \gamma\alpha\gamma^{-1}$ is a bijection from
  $\intiso(G)_{s(\gamma)}$ onto $\intiso(G)_{r(\gamma)}$. Each
  $\intiso(G)_u$ is a normal subgroup of $G^u_u$.
\end{lemma}

\section{A uniqueness theorem}\label{sec:abstractuniqueness}

The paper \cite{bnr:jfa14} presents a uniqueness theorem for the
$C^*$-algebras $C^*(\Lambda)$ of $k$-graphs $\Lambda$ that
characterises injectivity of homomorphisms induced by the universal
property. The hypotheses of this theorem are in terms of the abelian
subalgebra $M$ generated by elements $s_\mu s^*_\nu$ of $C^*(\Lambda)$
such that $\mu x = \nu x$ for every infinite path $x$ of $\Lambda$. As
discussed in Remark~4.11 of \cite{bnr:jfa14}, $M$ is the
completion of $\cc(\intiso(G_\Lambda)) \subseteq \cc(G_\Lambda)$,
where $G_\Lambda$ is the groupoid associated to $\Lambda$ as in
\cite{kumpas:nyjm00}. Here we use different methods to generalize the
uniqueness theorem of \cite{bnr:jfa14} to the reduced $C^*$-algebras
of Hausdorff \'etale groupoids. Our result characterises injectivity
of homomorphisms of $C^*_r(G)$ in terms of injectivity of their
restrictions to the canonical copy of $C^*_r(\intiso(G))$ in
$C^*_r(G)$.

To see that such a copy exists, note that $\intiso(G)$ is open in $G$ so there is an
injective $^*$-homomorphism $\iota : C_c(\intiso(G)) \hookrightarrow C_c(G)$ given by
extension by zero. Since $\iota$ is isometric for the respective $I$-norms, and by
\cite[Proposition~1.9]{Phil05}, this map $\iota$ extends to inclusions
\begin{equation}\label{eq:iota}
    \iota : C^*(\intiso(G)) \to C^*(G)\qquad\text{ and }\qquad
    \iota_r : C^*_r(\intiso(G)) \to C^*_r(G).
\end{equation}

\begin{thm}\label{thm:uniqueness}
  Let $G$ be a locally compact Hausdorff \'etale groupoid. Let
  $M := \iota(C^*(\intiso(G))) \subseteq C^*(G)$ and
  $M_r := \iota_r(C^*_r(\intiso(G))) \subseteq C^*_r(G)$.
  \begin{enumerate}
  \item\label{it:M state ext} Suppose that $u \in \go$ satisfies
    $G^u_u = \intiso(G)_u$.  If $\phi_r$ is a state of $M_r$ that
    factors through $C^*_r(G^u_u)$, then $\phi_r$ extends uniquely to
    $C^*_r(G)$.  If $\phi$ is a state of $M$ that factors through
    $C^*(G^u_u)$, then $\phi$ extends uniquely to $C^*(G)$.
  \item\label{it:uniqueness} If $\pi : C^*_r(G) \to D$ is a
    $C^*$-homomorphism, then $\pi$ is injective if and only if
    $\pi \circ \iota_r$ is an injective homomorphism of
    $C^*_r(\intiso(G))$.
  \end{enumerate}
\end{thm}

To prove the theorem, we need a few preliminary results. The first is
a slight improvement of the uniqueness theorem of \cite{bnr:jfa14} in
that we do not require that the subalgebra $M$ be abelian.

\begin{thm}\label{thm:abstractuniqueness}
  Let $A$ be a $C^*$-algebra and $M$ a $C^*$-subalgebra of
  $A$. Suppose that $S$ is a collection of states of $M$ such that
  \begin{enumerate}
  \item\label{it:ue} every $\phi\in S$ has a unique extension to a
    state $\tilde\phi$ of $A$; and
  \item\label{it:jf} the direct sum
    $\bigoplus_{\phi \in S} \pi_{\tilde\phi}$ of the GNS
    representations associated to extensions of elements of $S$ to $A$
    is faithful on $A$.
  \end{enumerate}
  Let $\rho : A \to B$ be a $C^*$-homomorphism. Then $\rho$ is
  injective if and only if it is injective on $M$.
\end{thm}
\begin{proof}
  The ``only if" statement is trivial. So suppose that $\rho$ is
  injective on $M$. Let $J=\ker\rho$; we must show that $J =
  \{0\}$.
  By hypothesis, we have $J\cap M=\{0\}$. Let $A_0 := J + M$; then
  $A_0$ is a $C^*$-subalgebra of $A$ by, for example,
  \cite[Corollary~1.8.4]{Dixmier}. Let $\gamma : A_0 \to M$ denote the
  quotient map. Since any state of $A_{0}$ extends to a state of $A$,
  hypothesis~(\ref{it:ue}) implies that each $\phi\in S$ has a unique
  state extension to $A_0$. Since $\phi \circ \gamma$ is an extension
  of $\phi$ to $A_0$, we deduce that $\phi \circ \gamma$ is the only
  extension of $\phi$ to a state of $A_0$ for each $\phi \in S$. Since
  $\tilde\phi|_{A_0}$ is also an extension of $\phi$ to $A_0$, we
  obtain
  \begin{equation}\label{eq:tphiA0}
    \tilde\phi(a)=\phi(\gamma(a))\quad\text{ for all $a\in A_0$}.
  \end{equation}

  Now fix $x \in J$. We have $a^*x^*xa\in J$ for all $a\in A$. Take
  $\phi \in S$. Since $J \subseteq A_0$, it follows
  from~\eqref{eq:tphiA0} that $\tilde\phi(a^*x^*xa)=0$ for all
  $a\in A$. Hence
  $(\pi_{\tilde\phi}(x) h \mid \pi_{\tilde\phi}(x) h) = 0$ for all
  $h \in \Hh_{\tilde\phi}$, giving $\tilde\phi(x) = 0$. Since
  $\phi \in S$ was arbitrary, we deduce that
  $\bigoplus_{\phi \in S} \pi_{\tilde\phi}(x) = 0$, and so $x = 0$
  by~(\ref{it:jf}).
\end{proof}

Next we need a technical lemma.

\begin{lemma}\label{lem:technical}
  Let $G$ be a locally compact Hausdorff \'etale groupoid.
  \begin{enumerate}
  \item\label{it:dense} The set
    $X: = \sset{u \in \go : G^u_u = \intiso(G)_u}$ is dense in $\go$.
  \item\label{it:cutdown} Suppose that $u \in \go$ satisfies
    $G^u_u = \intiso(G)_u$, and take $f \in C_c(G)$.  Then there
    exists $b \in C_c(\go)^+$ such that $\|b\| = b(u) = 1$ and
    $bfb\in C_c(\intiso(G))$.
  \end{enumerate}
\end{lemma}
\begin{proof}
  (\ref{it:dense}) We say that $B$ is an open \emph{nested} bisection
  of $G$ if there is a precompact open bisection $D$ of $G$ such that
  $\overline B\subset D$.  This forces $\overline{r(B)}\subset r(D)$
  because $r$ is a homeomorphism on $D$. Note that $G$ has a countable
  basis of open nested bisections.

  Fix an open nested bisection $B$ of $G$ with $\overline{B}\subset D$
  as above. Let $B' := B \cap \Iso(G) \setminus \intiso(G)$. We claim
  that $r(B')$ is nowhere dense in $\go$. To see this, suppose that
  $V \subset \overline{r(B')}$ is open. We show that $V$ is
  empty. Since $D$ is an open bisection, $r|_D$ is a homeomorphism
  onto $r(D)$.  Since $\overline{B'}\subset D$ we have
  \[
  r(VD)=V\subset \overline{r(B')}=r(\overline{B'}).
  \]
  Thus $VD=r^{-1}(V)\cap D$ is an open subset of
  $\overline{B'}\subset \Iso(G)-\Iso(G)^\circ$, which has empty
  interior. Therefore $V = \emptyset$.

  Since $G$ is \'etale, we have
  \begin{multline*}
    \sset{u \in \go : G^u_u \not= \intiso(G)_{u}} \\ = \sset{r(B \cap
      \Iso(G) \setminus \intiso(G)) : \text{$B$ is an open nested
        bisection}}.
  \end{multline*}
  Since $G$ is second countable, it follows from the preceding
  paragraph that $\sset{u \in \go : G^u_u \not= \intiso(G)_{u}}$ is a
  countable union of nowhere-dense sets, and hence nowhere dense by
  the Baire Category Theorem as stated in, for example,
  \cite[Theorem~6.34]{kel:general}. Hence
  $\sset{u \in \go : G^u_u = \intiso(G)_u}$ is dense in $\go$.

  (\ref{it:cutdown}) Fix $f \in C_c(G)$. Express
  $f = \sum_{D \in F} f_D$ where $F$ is a finite collection of
  precompact open bisections of $G$ and each $f_D \in C_c(D)$. Choose
  open neighbourhoods $\sset{V_D \subseteq \go : D \in F}$ of $u$ as
  follows:
  \begin{itemize}
  \item if $u = r(\alpha) = s(\alpha)$ for some $\alpha \in D$, take
    $V_D = r(D \cap \intiso(G)) = s(D \cap \intiso(G))$ so that
    $V_D D V_D \subseteq D \cap \intiso(G)$ (this $V_D$ is nonempty
    because $\alpha \in D \cap G^u_u \subset D \cap \intiso(G)$ by
    choice of $u$);
  \item if there exists $\alpha \in D$ such that $r(\alpha) = u$ and
    $s(\alpha) \not= u$ or $s(\alpha) = u$ and $r(\alpha) \not= u$,
    choose an open subset $D' \subset D$ containing $\alpha$ such that
    $r(D') \cap s(D') = \emptyset$, and take $V_D = r(D')$, so that
    $V_D D V_D = \emptyset$; and
  \item if $u \not\in r(D)$ and $u \not\in s(D)$, use that
    $f_D \in C_c(D)$ to choose a neighbourhood $V_D$ of $u$ such that
    $f_D|_{V_D D V_D} = 0$.
  \end{itemize}
  Let $V := \bigcap_{D \in F} V_D$. Then $V$ is open and contains
  $u$. Choose $b \in \cc(V)^+$ such that $b(u) = \|b\| = 1$. By
  construction, $\supp(bf_D b)$ is a compact subset of
  $V_D D V_D\subset \intiso(G)$, so $bf_D b\in C_c(\intiso(G))$. Thus
  $bfb = \sum_{D\in F} bf_D b$ is also in $C_c(\intiso(G))$.
\end{proof}

\begin{remark}
  Since $\go\subset \intiso(G)$, if $G^{u}_u=\{u\}$ then
  $u\in X=\{u: G^u_u=\intiso(G)_u\}$.  Now if $G$ is topologically
  principal then $\intiso(G)=\go$ (see \cite{bcfs:sf14}) so in this
  case $X=\{u: G^u_u=\{u\}\}$.
\end{remark}

\begin{lemma}
  \label{lem:statecutdown} Let $G$ be a locally compact Hausdorff
  \'etale groupoid and $u\in \go$ such that $G^u_u\subset \intiso(G)$.
  Let $\epsilon>0$ be given.
  \begin{enumerate}
  \item\label{it:redcutdown} Let $a\in C^*_r(G)$. Then there exist
    $b, c\in C^*_r(\intiso(G))$ such that $b$ positive of norm $1$,
    such that $\phi(b)=1$ for all states $\phi$ that factor through
    $C_r^*(G^u_u)$, and such that $\|bab-c\|_r<\epsilon.$
  \item\label{it:fullcutdown} Let $a\in C^*(G)$. Then there exist
    $b, c\in C^*(\intiso(G))$ such that $b$ positive of norm
    $1$, such that $\phi(b)=1$ for all states $\phi$ that factor through
    $C^*(G^u_u)$, and such that $\|bab-c\|<\epsilon.$
  \end{enumerate}
\end{lemma}
\begin{proof}
  We prove~(\ref{it:redcutdown}); the proof of~(\ref{it:fullcutdown})
  is exactly the same.  By continuity it suffices to show that for
  $a \in C_c(G)$ we can find $b \in C_c(\intiso(G))$ such that
  $bab \in C_c(\intiso(G))$ and $\phi(b) = \|b\|_r = 1$ for all states
  $\phi$ that factor through $C^*_r(G^u_u)$.

  Fix $f \in C_c(G)$. By Lemma~\ref{lem:technical}(\ref{it:cutdown}),
  there exists $b \in C_c(\go)^+ \subseteq C_c(\intiso(G))^+$ such
  that $b(u) = \|b\| = 1$ and $bfb\in C_r^*(\intiso(G))$. Since the
  quotient map from $C_r^*(\intiso(G))$ onto $C_r^*(G_u^u)$ carries
  $b$ to $1_{C_r^*(G^u_u)}$, we have $\phi(b) = \|b\| = 1$ for all
  states $\phi$ that factor through $C^*_r(G^u_u)$.
\end{proof}

We now have the wherewithal to prove Theorem~\ref{thm:uniqueness}.

\begin{proof}[Proof of Theorem~\ref{thm:uniqueness}]
  For~(\ref{it:M state ext}), we just prove the assertion about
  reduced $C^*$-algebras; the assertion about full $C^*$-algebras
  follows from exactly the same argument using
  part~(\ref{it:fullcutdown}) of Lemma~\ref{lem:statecutdown} instead
  of part~(\ref{it:redcutdown}). Fix $u \in \go$ such that
  $G^u_u \subseteq \intiso(G)$, and a state $\phi$ of
  $C_r^*(\intiso(G))$ that factors through $\cs_r(G^u_u)$. By the
  argument preceding \cite{and:tams79}*{Theorem~3.2}
  (\cite{and:tams79} is about unital $C^*$-algebras, but the argument
  also works in the non-unital setting) it will suffice to show that
  for each $a \in C^*_r(G)$ and $\varepsilon > 0$ there exists a
  positive element $b \in C_r^*(\intiso(G))$ such that
  $\phi(b) = \|b\| = 1$ and an element $c \in C_r^*(\intiso(G))$ such
  that $\|bab - c\| < \varepsilon$.  But this is just
  Lemma~\ref{lem:statecutdown}(\ref{it:redcutdown}).

  For~(\ref{it:uniqueness}), since $\iota_r$ is injective, the ``only
  if" is clear. Suppose that $\pi \circ \iota_r$ is injective, so
  $\pi$ is injective on $C_r^*(\intiso(G))$. Let
  $X=\{u\in \go: G^u_u = \intiso(G)_u\}$. For each $u\in X$, let $S_u$
  be the collection of pure states of $C_r^*(\intiso(G))$ that factor
  through $C_r^*(G^u_u)$. Let $S = \bigcup_{u\in X} S_u$. By
  part~(\ref{it:M state ext}) above, each $\phi$ in $S$ has unique
  extension $\tilde\phi$ to $C^*(G)$. For each $\phi \in S$, write
  $\pi_{\tilde\phi}$ for the GNS representation of $C^*_r(G)$
  associated to $\tilde\phi$ and for $A\subset S$, let
  $\pi_A=\bigoplus_{\phi\in A} \pi_{\tilde{\phi}}$.  By
  Theorem~\ref{thm:abstractuniqueness}, it suffices to show that
  $\pi_S:=\bigoplus_{\phi \in S} \pi_{\tilde\phi}$ is faithful on
  $C^*_r(G)$.

  Let $\Phi^r_I: C^*_{r}(\intiso(G))\to C_0(\go)$ be the conditional
  expectation that extends restriction of functions and $\ev_u$ be the
  evaluation map at $u$ for each $u\in X$.  We claim that for each
  $u\in X$, $\ev_u\circ \Phi^r_I$ factors through $C^*(G^u_u)$. Let
  $q_u: C^*_r(\intiso(G))\to C^*_r(G^u_u)$ be the quotient map and let
  $\mathcal{K}$ be an increasing net of compact subsets of
  $\go\setminus\{u\}$ such that
  $\bigcup \mathcal{K} = \go\setminus\{u\}$. For each
  $K \in \mathcal{K}$, choose $f_K \in C_c(\go \setminus\{u\})$ such
  that $f_K|_K \equiv 1$ and $0\leq f_K\leq 1$. Then
  $\{f_K\}_{K \in \mathcal{K}}$ is an approximate unit for
  $\ker(q_u)$. So for $a\in \ker(q_u)$, we have $a = \lim_K f_K a$,
  and hence
  \[
  \ev_u\circ \Phi_I^r(a) = \lim_K \ev_u \circ \Phi_I^r(f_K a) = \lim_K
  \ev_u(f_K) \ev_u\circ\Phi_I^r(a) = 0.
  \]
  This proves the claim.

  Suppose $a\in C^*_r(G)$ with $\pi_S(a)=0$.  We want to show that
  $a=0$. Let $\Phi_r: C^*_r(G)\to C_0(\go)$ be the faithful
  conditional expectation extending restriction of functions. Since
  $\Phi_r$ is faithful, it is enough to show that $\Phi_r(a^*a)=0$. By
  way of contradiction assume that $\Phi_r(a^*a)\neq 0$. By
  Lemma~\ref{lem:technical}, $X$ is dense, so there exists $u\in X$
  with $\Phi_r(a^*a)(u)>0$.  Pick $\epsilon$ such that
  \begin{equation}
    \label{eq:phi>e}
    \Phi_r(a^*a)(u)>\epsilon>0.
  \end{equation}
  By Lemma~\ref{lem:statecutdown}, there exists $b\in C^*_r(G)$ and
  $c \in C^*_r(\intiso(G))$ such that $\phi(b)=1$ for all states of
  $C^*_r(\intiso(G))$ that factor through $C^*(G^u_u)$ and such that
  \begin{equation}
    \label{eq:norm close}
    \|ba^*ab-c \|< \epsilon/2.
  \end{equation}
  We have $\pi_{S_u}(ba^*ab)=\pi_{S_u}(a^*a)=0$ by assumption. Thus
  \[
  \tilde{\phi}(ba^*ab)=\lim_\lambda \langle \pi_{\tilde{\phi}}(ba^*ab)
  (e_\lambda+N_{\tilde{\phi}}),e_\lambda+N_{\tilde{\phi}})\rangle= 0\]
  for all $\phi\in S_u$ where $\{e_\lambda\}$ is an approximate unit
  for $C^*_r(G)$. Now from Equation~\eqref{eq:norm close} we get
  $|\phi(c)|<\epsilon/2$ for all $\phi\in S_u$.  Thus
  $\|q_u(c)\|\leq \epsilon/2$.  Since $\ev_u\circ\Phi_I$ factors
  through $C^*(G^u_u)$ we deduce that
  $|\ev_u\circ\Phi_r(c) |\leq \epsilon/2$.

  On the other hand, if $\tilde{\psi}$ is the unique extension of a
  state $\psi$ on $C^*_r(\intiso(G))$ that factors through
  $C^*(G_u^u)$ then, since $\psi(b)=1$, we have
  $\tilde{\psi}(bdb)=\tilde{\psi}(d)$ for all $d\in C^*_r(G)$. In
  particular, $\ev_u\circ\Phi_r(ba^*a b)=\ev_u\circ\Phi_r(a^*a)$.
  Thus using Equation~\eqref{eq:norm close} again, we have
  \[
  |\ev_u\circ\Phi_r(a^*a)-\ev_u\circ\Phi_r(c)| =
  |\ev_u\circ\Phi_r(ba^*ab)-\ev_u\circ\Phi_r(c)| < \epsilon/2.
  \]
  Hence
  $|\ev_u\circ\Phi_r(a^*a)| < \epsilon/2 + \epsilon/2 =
  \epsilon$. This contradicts Equation~\eqref{eq:phi>e}.
\end{proof}

\section{Maximal abelian subalgebras, Cartan subalgebras and
  pseudo-diagonals}\label{sec:abelian isotropy}

In \cite{bnr:jfa14}*{Remark~4.11}, the authors conjecture that if
$\Lambda$ is a $k$-graph, and $M$ is the $C^*$-subalgebra of
$C^*(\Lambda)$ spanned by the elements $s_\mu s^*_\nu$ such that
$\mu x = \nu x$ for every infinite path $x$, then $M$ is a maximal
abelian subalgebra of $C^*(\Lambda)$. Yang established this in
\cite{yang:xx14} for cofinal $k$-graphs, which are those whose
groupoids are minimal. She has recently communicated to us the
preprint \cite{yang:xxxx} in which she proves the same result for
those $k$-graphs with the property that the interior of the isotropy
in the associated groupoid is closed.

We will show more generally that if $G$ is an \'etale groupoid $G$ in
which the interior of the isotropy is abelian, and if either~(a) the
interior of the isotropy is closed in $G$, or~(b) $G$ carries a
continuous cocycle into a discrete abelian group $H$ that is injective
on the isotropy over every unit of $G$, then the subalgebra $M_r$ of
Theorem~\ref{thm:uniqueness} is a maximal abelian subalgebra of
$C^*_r(G)$.

We also investigate when $M_r$ is a Cartan subalgebra in the sense of
Renault and a pseudo-diagonal in the sense
of~\cite{nagrez:pams14}. Both conditions require the existence of a
conditional expectation $\Psi : C^*_r(G) \to M_r$. We prove that such
an expectation exists if and only if $\intiso(G)$ is closed in $G$. We
show by example that there exist $k$-graphs $\Lambda$ for which the
interior of the isotropy in the associated infinite-path groupoid is
not closed, and therefore for which $M_r$ is not a Cartan subalgebra
of $C^*(\Lambda)$. We tie up a loose thread from~\cite{nagrez:jlms12}
by proving that every Cartan subalgebra is a pseudo-diagonal.

Throughout this section we will make frequent use of the following
fact: if $G$ is a second-countable \'etale Hausdorff groupoid, then
\cite{ren:groupoid}*{Proposition~II.4.2(i)} implies that the injection
$j:C_{c}(G)\to C_{0}(G)$ extends to an injective norm-decreasing
linear map $j: \cs_{r}(G) \to C_{0}(G)$.

\begin{prop} \label{prp:M FCE} Let $G$ be a locally compact Hausdorff
  \'etale groupoid. 
Let $M$ and $M_{r}$ be
  as in
  Theorem~\ref{thm:uniqueness}.
  \begin{enumerate}
  \item There exists a conditional expectation
  from $C^*_r(G)$ to $M_r$ if and only if $\intiso(G)$ is closed in
  $G$.
\item If $\intiso(G)$ is closed, then there is a faithful
  conditional expectation
  \begin{equation*}
    \Psi_r : C^*_r(G) \to M_r
  \end{equation*}
 such that
  $\Psi_r(f) = \iota_r(f|_{\intiso(G)})$ for all $f \in C_c(G)$.
\item If
  $\intiso(G)$ is closed and amenable, then there is also a
  conditional expectation (not necessarily faithful)
  $\Psi : C^*(G) \to M$ satisfying $\Psi(f) = \iota(f|_{\intiso(G)})$
  for all $f \in C_c(G)$.
  \end{enumerate}
\end{prop}

To prove this proposition we need a lemma.

\begin{lemma}
  \label{lem:inj}
  If $\intiso(G)$ is amenable then $\iota: C^*(\intiso(G))\to C^*(G)$
  is injective.
\end{lemma}

\begin{proof}
  Observe that the conditional expectations
  $\Phi_I : C^*(\intiso(G)) \to C_0(\go)$ and
  $\Phi : C^*(G) \to C_0(\go)$ determined by restriction of functions
  satisfy $\Phi \circ \iota = \iota \circ \Phi_I$. Restriction of
  functions also determines a faithful conditional expectation
  $\Phi^r_I : C^*_r(\intiso(G)) \to C_0(\go)$ . Since $\intiso(G)$ is
  amenable, $C^*(\intiso(G)) = C^*_r(\intiso(G))$. So
  $\Phi_I=\Phi^r_I$ and the latter is faithful. So a standard argument
  (see, for example, \cite{sww:dm14}*{Lemma~3.13}) shows that $\iota$
  is injective.
\end{proof}

\begin{proof}[Proof of Proposition~\ref{prp:M FCE}]
  We first show that if $\intiso(G)$ is not closed, then there does
  not exist a conditional expectation from $C^*_r(G)$ to $M_r$. We
  argue by contradiction: that is, we suppose that $\intiso(G)$ is not
  closed and that there is a conditional expectation
  $\Phi : C^*_r(G) \to M_r$, and we derive a contradiction. Fix
  $\gamma \in \overline{\intiso(G)} \setminus \intiso(G)$, and choose
  a sequence $\gamma_n$ in $\intiso(G)$ converging to $\gamma$.
  Choose a precompact open bisection $U$ containing $\gamma$ and a
  function $f \in C_c(G)$ such that $\supp(f)\subset U$ and
  $f(\gamma) = 1$. Without loss of generality, the $\gamma_n$ are all
  in the support of $f$. Since each $\gamma_n$ is interior to the
  isotropy, we can choose open sets $V_n \subseteq U \cap \intiso(G)$
  such that $\gamma_n \in V_n$. The sets $r(V_n) \subseteq \go$ are
  then open sets of units, and we can choose functions
  $g_n \in C_c(r(V_n))$ such that $g_n(r(\gamma_n)) = 1$.

  Now each $g_n \in M_r$, and so $\Phi(g_n * f) = g_n * \Phi(f)$ for
  each $n$. Since $f$ is supported on a bisection and $g_n$ is
  supported in $G^{(0)}$, the element $g_n * f$ is supported on
  $\supp(f) \cap r^{-1}(\supp(g_n)) \subseteq \supp(f) \cap
  r^{-1}(r(V_n)) \subseteq \intiso(G)$.
  Since $g_n$ and $f$ are compactly supported, we obtain
  $g_n * f \in C_c(\intiso(G)) \subseteq M_r$. Hence
  $\Phi(g_n * f) = g_n * f$ for each $n$. We have
  $j(g_n * h)(\gamma_n) = j(h)(\gamma_n)$ for all $h \in C^*_r(G)$,
  and so
  \[
  j(\Phi(f))(\gamma_n) = j(g_n * \Phi(f))(\gamma_n) = j(\Phi(g_n *
  f))(\gamma_n) = j(g_n * f)(\gamma_n) = f(\gamma_n).
  \]
  Since $f$ is continuous and $f(\gamma) = 1$, we have
  $j(\Phi(f))(\gamma_n) \to 1$.

  We have $\Phi(f) \in M_r$, say $\Phi(f) = \iota_r(a)$ where
  $a \in C^*_r(\intiso(G))$. We can choose elements
  $h_n \in C_c(\intiso(G))$ such that $h_n \to a$ in
  $C^*_r(\intiso(G))$. We therefore have $\iota_r(h_n) \to \Phi(f)$ in
  $C^*_r(G)$. Since $j : C^*_r(G) \to C_0(G)$ is continuous, it
  follows that $j(\iota_r(h_n)) \to j(\Phi(f))$ in $C_0(G)$, and in
  particular $j(\iota_r(h_n))(\gamma) \to j(\Phi(f))(\gamma)$. Since
  each $h_n$ belongs to $C_c(\intiso(G))$ and since
  $\gamma \not\in \intiso(G)$, we have $j(\iota_r(h_n))(\gamma) = 0$
  for all $n$ and it follows that $j(\Phi(f))(\gamma) = 0$.

  Putting all this together, we have $\gamma_n \to \gamma$,
  $j(\Phi(f))(\gamma_n) \to 1$ and $j(\Phi(f))(\gamma) = 0$, which
  contradicts continuity of $j(\Phi(f))$.  This proves the ``only if''
  part of (a).

  Now suppose that $\intiso(G)$ is closed. Then
  $f|_{\intiso(G)}\in C_c(\intiso(G))$ for all $f\in C_c(G)$.  Define
  $\Psi_0 : C_c(G) \to C_c(\intiso(G))$ by
  $\Psi_0(f) = \iota(f|_{\intiso(G)})$.  The image of $\Psi_0$ is
  $C_c(\intiso(G))$ since $\intiso(G)$ is closed. Also, $\Psi_0$ is a
  linear idempotent. We claim that
  \[
  \|\Psi_0(f)\|_{\cs_r(\intiso(G))} \le \|f\|_{\cs_r(G)}\quad\text{
    for all $f \in C_c(G)$.}
  \]
  Fix $f \in C_c(G)$ and $\varepsilon > 0$. We will show that
  $\|\Psi_0(f)\|_r \le \|f\|_r + \varepsilon$. Let
  $f_0 := f|_{\intiso(G)}$. There is a unit $u \in \go$ such that the
  associated regular representation
  $L^u : C^*(\intiso(G)) \to \Bb(\ell^2(\intiso(G)_u))$ satisfies
  $\|L^u(f_0)\| \ge \|f_0\|_r - \varepsilon$. Let $\pi_u$ be the
  regular representation of $C^*(G)$ on $\ell^2(G_u)$. Let
  $P \in \Bb(\ell^2(G_u))$ be the orthogonal projection into
  $\clsp\sset{\delta_\gamma : \gamma \in \intiso(G)_u} \subset
  \ell^2(G_u)$.

  For $\alpha,\beta \in G_u$, we have
  \begin{align*}
    \big(P\pi_u(f) P \delta_\alpha \mid \delta_\beta \big)
    &= \begin{cases} \big(\pi_u(f)\delta_\alpha \mid
      \delta_\beta\big)
      &\text{ if $\alpha,\beta \in \intiso(G)$}\\
      0 &\text{ otherwise}
    \end{cases} \\
    &= \begin{cases}
      f(\beta\alpha^{-1}) &\text{ if $\alpha,\beta \in \intiso(G)$}\\
      0 &\text{ otherwise}
    \end{cases}\\
    &= \begin{cases} \big(L^u(f_0) \delta_\alpha \mid
         \delta_\beta\big) &\text{ if
           $\alpha,\beta \in \intiso(G)$}\\
         0 &\text{ otherwise.}
       \end{cases}
  \end{align*}
  So the canonical unitary isomorphism
  $\ell^2(\intiso(G)_u) \cong P \ell^2(G_u)$ intertwines $P\pi_u(f) P$
  and $L^u(f_0)$, giving $\|L^u(f_0)\| = \|P\pi_u(f) P\|$. Hence
  \[
  \|\Psi_0(f)\|_r = \|\iota_r(f_0)\|_r \le \|L^u(f_0)\| + \varepsilon
  = \|P\pi_u(f) P\| + \varepsilon \le \|f\|_r + \varepsilon.
  \]
  Hence $\Psi_0$ extends to a linear idempotent
  $\Psi_r : C^*_r(G) \to M_r$.  Theorem~II.6.10.2 of
  \cite{bla:operator06} shows that $\Psi_r$ is a conditional
  expectation. Since $\go\subset \intiso(G)$, the canonical
  expectation $\Phi_r : C^*_r(G) \to C_0(\go)$ satisfies
  $\Phi_r = \Phi_r \circ \Psi_r$. Since $\Phi_r$ is faithful, it
  follows that $\Psi_r$ is too.

This gives us the remaining implication for part~(a) as well as
part~(b).

To establish (c), suppose that $\intiso(G)$ is amenable. Then
Lemma~\ref{lem:inj} shows $\iota$ is injective and hence isometric,
and so we have
$\|\iota(f)\| = \|f\|_{C^*(\intiso(G))} =\|f\|_{C^*_r(\intiso(G))}$
for all $f \in \cc(\intiso(G))$. In particular we saw above that for
$f \in \cc(G)$ we have
$\|\iota(f|_{\intiso(G)})\| \le \|f\|_{C^*_r(G)} \le \|f\|_{C^*(G)}$.
Hence $\Psi_0$ extends to a linear idempotent $\Psi$ of norm 1 from
$C^*(G)$ to $M$. Once again, \cite[Theorem~II.6.10.2]{bla:operator06}
shows that $\Psi$ is a conditional expectation.
\end{proof}

We now consider when $M_r$ is a maximal abelian subalgebra of
$C^*_r(G)$.

\begin{thm}\label{thm:masa}
  Let $G$ be a locally compact Hausdorff \'etale groupoid, and suppose
  that $\intiso(G)$ is abelian. Then $\intiso(G)$ is amenable,
  $M := \iota(C^*(\intiso(G)))$ is an abelian subalgebra of $C^*(G)$,
  and $M_r = \iota_r(C^*(\intiso(G)))$ is an abelian subalgebra of
  $C^*_r(G)$. Suppose that either
  \begin{enumerate}
  \item \label{masa-a}$\intiso(G)$ is closed in $G$, or
  \item \label{masa-b}there exist a countable discrete abelian group $H$ and a
    continuous 1-cocycle $c : G \to H$ such that $c|_{G^x_x}$ is
    injective for each $x \in \go$.
  \end{enumerate}
  Then $M_r$ is maximal abelian in $C^*_r(G)$.
\end{thm}

Before proving Theorem~\ref{thm:masa}, we establish a technical result
that may be useful in future.

\begin{lemma}\label{lem:A(P)masa}
  Let $G$ be a locally compact Hausdorff \'etale groupoid, and suppose
  that $\intiso(G)$ is abelian. Then $M_r$ is maximal abelian in
  $C^*_r(G)$ if and only if
  \[
  \{a \in C^*_r(G) : j(a) \in C_0(\intiso(G))\} \subseteq M_r.
  \]
\end{lemma}
\begin{proof}
  Let $A := \{a \in C^*_r(G) : j(a) \in C_0(\intiso(G))\}$. Since $j$
  is continuous, $A$ is closed, and hence a $C^*$-subalgebra of
  $C^*_r(G)$. Renault proves that $j$ is multiplicative for the usual
  convolution product, and so $A$ is an abelian $C^*$-subalgebra of
  $C^*_r(G)$ containing $M_r$. So if $M_r$ is maximal abelian, then it
  is equal to $A$.

  Conversely suppose that $M_r = A$. Suppose that $a \in C^*_r(G)$
  commutes with every element of $M_r$. We must show that $a \in M_r$.
  Fix $\alpha \in G\setminus \Iso(G)$.  Since
  $r(\alpha) \not= s(\alpha)$ there exists $b \in C_0(\go)$ such that
  $b(s(\alpha)) = 1$, and $b(r(\alpha)) = 0$. Since $ab = ba$, it
  follows from \cite{ren:groupoid}*{Proposition~II.4.2(iii)} that
  \[
  0 = \big|j(ab - ba)(\alpha)\big| = \big|j(a)(\alpha)b(s(\alpha)) -
  b(r(\alpha))j(a)(\alpha)\big| = |j(a)(\alpha)|.
  \]
  So $j(a)$ vanishes on $G \setminus \Iso(G)$. Now for
  $\alpha \in G \setminus \intiso(G)$, there is a sequence
  $\alpha_n \in G \setminus \Iso(G)$ with $\alpha_n \to \alpha$. Since
  $j(a)$ is continuous, we deduce that
  $j(a)(\alpha) = \lim j(a)(\alpha_n) = 0$.
\end{proof}

\begin{proof}[Proof of Theorem~\ref{thm:masa}]
  Since $\intiso(G)$ is an abelian-group bundle,
  \cite{ren:xx13}*{Theorem~3.5} shows that it is amenable. Since
  $\intiso(G)$ is abelian, $M$ and $M_r$ are abelian.

  Now fix $a \in C^*_r(G)$ such that $j(a) \in C_0(\intiso(G))$. By
  Lemma~\ref{lem:A(P)masa}, to complete the proof it suffices to show
  that if either (\ref{masa-a})~or~(\ref{masa-b}) holds, then $a\in M_r$.

  First suppose that~(\ref{masa-a}) holds. Then Proposition~\ref{prp:M
    FCE} shows that restriction of compactly supported functions
  extends to a conditional expectation $\Psi_r: C^*_r(G)\to M_r$. Take
  $a_{n}\in C_{c}(G)$ with $a_{n}\to a$ in $\cs(G)$. We have
  $a_{n}-\Psi_r(a_{n})\to a-\Psi_r(a)$ in $\cs_r(G)$. Therefore
  $j\bigl(a_{n}-\Psi_r(a_{n})\bigr)\to j\bigl(a-\Psi_r(a)\bigr)$.  But
  $j$ is the identity on $C_{c}(G)$ and each $a_{n}-\Psi_r(a_{n})$
  vanishes on $\intiso(G)$.  Hence $j(a-\Psi_r(a))$ vanishes on
  $\intiso(G)$.  But then $j(a-\Psi_r(a))=0$.  Since $j$ is injective,
  we obtain $a = \Psi_r(a)$. Since $\Psi_r(\cs(G)) = M_r$, we deduce
  that $a \in M_r$ as required.

  Now suppose that~(\ref{masa-b}) holds. Let
  $A := \{b \in C^*_r(G) : j(b) \in C_0(\intiso(G))\}$. Let
  $\gamma : \widehat{H} \to \Aut(C^*_r(G))$ be the action such that
  $\gamma_\chi(f)(\alpha) = \chi(c(\alpha))f(\alpha)$ for
  $f \in C_c(G)$. By continuity, for $b \in C^*_r(G)$,
  $\chi \in \widehat{H}$ and $\alpha \in G$, we have
  \[
  j(\gamma_\chi(b))(\alpha) = \chi(c(\alpha)) j(b)(\alpha).
  \]
  Hence $A$ is invariant under $\gamma$. In particular, the Fourier
  coefficients
  \[
  \Phi_h(a) := \int_{\widehat{H}} \overline{\chi(h)}
  \gamma_\chi(a)\,d\mu(\chi)
  \]
  of $a$ belong to $A$. Since $H$ is amenable, $a$ can be expressed as
  a norm-convergent sum of the $\Phi_h(a)$ (see, for example, the
  argument of \cite[Theorem~5.6]{bc:jfaa15}), so we just need to show
  that each $\Phi_h(a) \in M_r$.

  Fix $h \in H$ and consider $a_h := \Phi_h(a)$. Fix
  $\varepsilon > 0$. It suffices to show that there exists
  $b \in C_c(\intiso(G))$ such that $\|a_h - b\|_r \le
  \varepsilon$.
  For this, observe that
  $\supp(j(a_h)) = \supp(a) \cap c^{-1}(h) \subseteq
  \overline{\intiso(G)} \cap c^{-1}(h)$.
  Since $c$ is injective on each $G^x_x$, this set is a bisection. The
  $I$-norm is finite and agrees with the supremum norm on $C_0(V)$ for
  any bisection $V$. Since the $I$-norm dominates the full norm, and
  hence the reduced norm, on $C_c(G)$, it follows that for
  $b \in j\inv(C_0(\{\alpha : j(a_h)(\alpha) \not= 0\}))$, we have
  $\|j(b)\|_\infty = \|j(b)\|_I \ge \|b\|_r \ge \|j(b)\|_\infty$,
  giving equality throughout. Thus while it is not true for arbitrary
  $a \in C_c(G)$ that $\|j(a)\|_\infty=\|a\|_r$, we do have
  \begin{equation}
    \label{eq norm eq}
    \|j(b)\|_\infty=\|b\|_r\quad \text{for all $b \in
    j\inv(C_0(\set{\alpha : j(a_h)(\alpha) \not= 0}$})).
  \end{equation}

  Fix $\varepsilon > 0$. The sets
  $U_\varepsilon := \{\alpha \in G : |j(a_h)(\alpha)| \ge
  \varepsilon\}$
  and
  $U_{\varepsilon/2} := \{\alpha \in G : |j(a_h)(\alpha)| \ge
  \varepsilon/2\}$
  are compact subsets of
  $\{\alpha : j(a_h)(\alpha) \not= 0\} \subseteq \intiso(G)$. Since
  $\overline{\intiso(G)} \cap c^{-1}(h)$ is a bisection, there exists
  $g \in C_c(\go, [0,1])$ such that $g(r(\alpha)) = 1$ for all
  $\alpha \in U_\varepsilon$, and
  $\supp(g) \subseteq r(U_{\varepsilon/2})$. We have
  $g * a_h = g * j(a_h) \in C_c(\intiso(G))$. Furthermore,
  $a_h - g * a_h \in C_0(\{\alpha : j(a_h)(\alpha) \not= 0\})$. By the
  preceding paragraph, we therefore have
  \[
  \|a_h - g * a_h\|_r = \|j(a_h - g * a_h)\|_\infty = \|j(a_h) - g *
  a_h\|_\infty.
  \]
  By construction of $g$, we have $(g * a_h)(\alpha) = j(a_h)(\alpha)$
  whenever $|a_h(\alpha)| \ge \varepsilon$, and
  $|j(a_h - g * a_h)(\alpha)| = (1 - g(r(\alpha)))|a_h(\alpha)| <
  |a_h(\alpha)|$
  for all $\alpha$. Thus $\|j(a_h) - g * a_h\|_\infty \le \varepsilon$
  as required.
\end{proof}

The previous theorem resolves some issues left unanswered in \cite{ren:imsb08} and
\cite{nagrez:pams14}. For the next result, observe that when $\intiso(G)$ is abelian, the
homomorphisms $\iota$ and $\iota_r$ of~\eqref{eq:iota} are injective by
Lemma~\ref{lem:inj} and \cite[Proposition~1.9]{Phil05} respectively.

\begin{cor}\label{cor:pseudodiagonal}
  Let $G$ be a locally compact Hausdorff \'etale groupoid and suppose
  that $\intiso(G)$ is abelian. The following are equivalent:
  \begin{enumerate}
  \item\label{it:pd} $\iota_r(C^*(\intiso(G))) \subset C^*_r(G)$ is a
    pseudo-diagonal in the sense of \cite{nagrez:pams14}*{page~268};
  \item\label{it:csa} $\iota_r(C^*(\intiso(G))) \subset C^*_r(G)$ is a
    Cartan subalgebra in the sense of
    \cite{ren:imsb08}*{Definition~4.5}; and
  \item\label{it:iic} $\intiso(G)$ is closed in $G$.
  \end{enumerate}
  In particular, if $G$ is amenable and $\intiso(G)$ is closed, then
  $\iota(C^*(\intiso(G))) \subseteq C^*(G)$ is both a pseudo-diagonal
  and a Cartan subalgebra.
\end{cor}
\begin{proof}
  Both (\ref{it:pd})~and~(\ref{it:csa}) imply by definition that there
  is a conditional expectation from $C^*_r(G)$ to
  $\iota_r(C^*(\intiso(G)))$, and then the ``only if" implication in
  the first statement of Theorem~\ref{prp:M FCE}
  gives~(\ref{it:iic}). So it suffices to prove that~(\ref{it:iic})
  implies (\ref{it:pd})~and~(\ref{it:csa}).

  Suppose, then, that $\intiso(G)$ is closed. Let
  $M := \iota_r(C^*(\intiso(G)))$.  Theorem~\ref{thm:masa} implies
  that $M$ is a maximal abelian subalgebra of $C^*_r(G)$.  Let
  $S_{\intiso(G)}$ be the set of pure states of $M$ that factor
  through $C^*_r(G^u_u)$ for some unit $u$ with
  $G^u_u = \intiso(G)_u$. We claim that $S_{\intiso(G)}$ is dense in
  the set of all pure states of $M$; that is, the corresponding set
  $S_{\intiso(G)}^{\wedge}$ is dense in the Gelfand dual $\widehat M$
  of $M$. By \cite{mrw:tams96}*{Corollary~3.4} and the subsequent
  remarks, and by \cite{mrw:tams96}*{Proposition~3.6}, the map
  $p:\widehat M\to\go$ is an open map making $\intiso(G)$ into an
  abelian-group bundle over $\go$. By Lemma~\ref{lem:technical}
  part~(\ref{it:dense}), it suffices to show that if $D\subset\go$ is
  dense then $p^{-1}(D)$ is dense in $\widehat M$. To see this, fix
  $\sigma\in\widehat M$.  There exist $u_{n}\in D$ such that
  $u_{n}\to p(\sigma)$.  Since $p$ is open, we can invoke
  \cite{wil:crossed}*{Proposition~1.15}, pass to a subsequence and
  relabel so that there exist $\sigma_{n}\in\widehat M$ such that
  $\sigma_{n}\to \sigma$ and $p(\sigma_{n})=u_{n}$.  This suffices and
  the claim is established.

  Now Theorem~\ref{thm:uniqueness}(\ref{it:M state ext}) implies that
  $S_{\intiso(G)}$ is a weak-$*$ dense set of pure states $\phi$ of
  $M$ for which $\phi \circ \Psi$ is the unique extension of $\phi$ to
  a state of $C^*(G)$. In the terminology of
  \cite{nagrez:pams14}*{page~266} (the definition just below
  Remark~2.4), we have just established that $M \subset C^*(G)$ has
  the canonical almost extension property with associated expectation
  $\Psi$. Proposition~\ref{prp:M FCE} shows that $\Psi$ is a faithful
  conditional expectation, and so $M$ is a pseudo-diagonal as defined
  in \cite{nagrez:pams14}*{p.~268}.

  To see that $M$ is also a Cartan subalgebra, we have to check that
  it is a regular maximal abelian subalgebra containing an approximate
  identity for $C^*(G)$ and admitting a faithful conditional
  expectation. It contains an approximate identity because
  $C_0(\go) \subset M$ does. We have already checked that it is a
  maximal abelian subalgebra and admits a faithful conditional
  expectation. For regularity, we must show that
  $\sset{n \in C^*(G) : n^* M n \cup n M n^* \subset M}$ generates
  $C^*(G)$ as a $C^*$-algebra. For this, observe that if
  $a \in C_c(\intiso(G))$ is supported in an open bisection
  $U \subset \intiso(G)$ and $n \in C_c(G)$ is supported in an open
  bisection $B$ in $G$, then $n^* a n$ is supported in $B^{-1} U B$
  and $n a n^*$ is supported in $B U B^{-1}$. Since $\intiso(G)$ is
  invariant under conjugation in $G$, it follows that both $n^* a n$
  and $n a n^*$ belong to $C_c(\intiso(G))$. Now continuity and
  linearity shows that if $a \in M$ and $n \in C_c(G)$ is supported on
  an open bisection, then $n^* a n, n a n^* \in M$.  Since
  $\sset{n \in C_c(G) : n\text{ is supported on an open bisection}}$
  generates $C^*(G)$, we deduce that $M$ is regular, and hence Cartan.
\end{proof}

In particular, we obtain from the above a complete answer to the question asked in
\cite{bnr:jfa14}*{Remark~4.11}. For background and notation for $k$-graphs and their
infinite-path spaces, see~\cite{kumpas:nyjm00}. For our purposes it suffices to recall
that each $k$-graph $\Lambda$ (with degree map $d : \Lambda \to \N^k$) has an
infinite-path space $\Lambda^\infty$, that if $x \in \Lambda^\infty$ and $\lambda \in
\Lambda r(x)$, then we can form the infinite path $\lambda x$, and that the $k$-graph
groupoid $G = G_\Lambda$ consists of triples of the form $(\mu x, d(\mu) - d(\nu), \nu
x)$ where $x \in \Lambda^\infty$ and $\mu,\nu \in \Lambda r(x)$.

\begin{cor}[Yang \cite{yang:xxxx}]\label{cor:k-graph}
  Let $\Lambda$ be a row-finite $k$-graph with no sources, and let $G$
  be the groupoid associated to $\Lambda$ in
  \cite{kumpas:nyjm00}. Then
  $M := \clsp\sset{s_\mu s^*_\nu : \mu x = \nu x \text{ for every } x
    \in \Lambda^\infty}$
  is a maximal abelian subalgebra of $C^*(\Lambda)$. The following are
  equivalent:
  \begin{enumerate}
  \item $M$ is a pseudo-diagonal;
  \item $M$ is a Cartan subalgebra;
  \item the set $\sset{(\mu x, d(\mu) - d(\nu), \nu x) \in G : \mu y = \nu y \text{
      for all } y \in r(x)\Lambda^\infty}$ is closed.
  \end{enumerate}
\end{cor}
\begin{proof}
  Let $G$ be the groupoid associated to $\Lambda$ in
  \cite{kumpas:nyjm00}.  As observed in Remark~4.11 of
  \cite{bnr:jfa14}, the isomorphism $C^*(\Lambda) \cong C^*(G)$ of
  \cite{kumpas:nyjm00}*{Corollary~3.5} carries
  $\clsp\sset{s_\mu s^*_\nu : \mu y = \nu y \text{ for every } y \in
    s(\mu)\Lambda^\infty}$
  to $\iota(C^*(\intiso(G))) \subset C^*(G)$. The canonical cocycle
  $(x,m,y) \mapsto m$ from $G$ to $\Z^k$ is injective on each $G^x_x$,
  so the result follows from Theorem~\ref{thm:masa} and
  Corollary~\ref{cor:pseudodiagonal}.
\end{proof}

The preceding result begs the question: is the interior of the
isotropy always closed in the infinite-path groupoid of a $k$-graph?
The answer when $k=1$ is ``yes", as can be deduced from
Proposition~\ref{prp:M FCE} and
\cite{nagrez:jlms12}*{Theorem~3.6}. However, this happy situation does
not persist for $k \ge 2$, as the next example shows.

\begin{example}
  \begin{figure}[t]
    \centering
    \[
  \begin{tikzpicture}[scale=1.5, >=stealth, decoration={markings,
      mark=at position 0.5 with {\arrow{>}}}]
    \node[circle, inner sep=1pt, fill=black] (v) at (0,0) {};
    \node[circle, inner sep=1pt, fill=black] (u) at (2,1) {};
    \node[circle, inner sep=1pt, fill=black] (w) at (2,-1) {};
    \draw[blue, postaction=decorate] (v) .. controls +(0.75,0.75) and
    +(-0.75,0.75) .. (v) node[above, pos=0.5, black]
    {\small$e_b$}; \draw[blue, postaction=decorate] (u) .. controls
    +(0.75,0.75) and +(-0.75,0.75) .. (u) node[above, pos=0.5, black]
    {\small$f_b$}; \draw[blue, postaction=decorate] (w) .. controls
    +(0.7,0.7) and +(-0.65,0.65) .. (w) node[circle, inner sep=0.1pt,
    above, pos=0.5, black]
    {\small$g_b$}; \draw[blue, postaction=decorate] (w) .. controls
    +(1,1) and +(-1,1) .. (w) node[circle, inner sep=0.1pt, above,
    pos=0.5, black]
    {\small$h_b$}; \draw[red, dashed, postaction=decorate] (v)
    .. controls +(0.75,-0.75) and +(-0.75,-0.75) .. (v) node[below,
    pos=0.5, black]
    {\small$e_r$}; \draw[red, dashed, postaction=decorate] (u)
    .. controls +(0.75,-0.75) and +(-0.75,-0.75) .. (u) node[below,
    pos=0.5, black]
    {\small$f_r$}; \draw[red, dashed, postaction=decorate] (w)
    .. controls +(0.65,-0.65) and +(-0.65,-0.65) .. (w) node[circle,
    inner sep=0.1pt, below, pos=0.5, black]
    {\small$g_r$}; \draw[red, dashed, postaction=decorate] (w)
    .. controls +(1,-1) and +(-1,-1) .. (w) node[circle, inner
    sep=0.1pt, below, pos=0.5, black] {\small$h_r$};
    \draw[blue, postaction=decorate, out=200, in=40] (u) to
    node[above, pos=0.5, black]
    {$\alpha_b$} (v); \draw[blue, postaction=decorate, out=140,
    in=340] (w) to node[above, pos=0.5, black]
    {$\beta_b$} (v); \draw[red, dashed, postaction=decorate, out=220,
    in=20] (u) to node[below, pos=0.5, black]
    {$\alpha_r$} (v); \draw[red, dashed, postaction=decorate, out=160,
    in=320] (w) to node[below, pos=0.5, black] {$\beta_r$} (v);
  \end{tikzpicture}
  \]
    \caption{A $2$-graph such that $\intiso(G)$ is not closed.}
    \label{fig:2-graph}
  \end{figure}

  Consider the 2-coloured graph in Figure~\ref{fig:2-graph} where the
  factorisation rules are given by
  \begin{gather*}
    e_b \alpha_r = e_r \alpha_b,\ e_b\beta_r = e_r\beta_b,\
    \alpha_bf_r = \alpha_r f_b,\ f_b f_r = f_r f_b,\
    \beta_b g_r = \beta_r g_b,\ \beta_b h_r = \beta_r h_b,\\
    g_bg_r = g_r g_b,\ g_b h_r = h_rg_b,\ h_bg_r = g_rh_b,\text{ and }
    h_b h_r = h_r h_b.
  \end{gather*}

  Let $\Lambda$ be the resulting $2$-graph and $w=r(g_r)$. By
  construction, $w\Lambda$ is isomorphic to the $2$-graph
  $B_2 \times B_2$ where $B_2$ is the bouquet of two loops. In
  particular, $w\Lambda$ is aperiodic. Fix an infinite path
  $y_0 \in w\Lambda^\infty$ such that $\sigma^m(y_0) = \sigma^n(y_0)$
  only when $m = n \in \N^2$. Put
  \[
  y := \beta_b y_0.
  \]
  So $\sigma^m(y) = \sigma^n(y)$ implies $m = n$, and
  $r(y) = r(\beta_b)=:v$.

  Note that
  $Z(\alpha_b) = Z(\alpha_r) = \{\alpha_b (f_bf_r)^\infty\} =
  \{\alpha_r (f_bf_r)^\infty\}$. Also
  \[
  \sigma^{(1,0)}(\alpha_b (f_bf_r)^\infty) = (f_bf_r)^\infty =
  \sigma^{(0,1)}(\alpha_b (f_bf_r)^\infty),
  \]
  and we deduce that $Z(\alpha_b, \alpha_r)$ is contained in
  $\intiso(G_\Lambda)$. Let
  \[
  x := \alpha_b (f_bf_r)^\infty.
  \]
  Also note that
  \[
  z := (e_be_r)^\infty
  \]
  satisfies $\sigma^{(1,0)}(z) = \sigma^{(0,1)}(z)$.

  For each $n$, let
  \[
  \gamma_n := \big((e_b e_r)^n e_b x, (1,-1), (e_be_r)^n e_r x\big)
  \in G_\Lambda.
  \]
  Using that $Z(\alpha_b, \alpha_r) = \{(x, (1,-1), x)\}$, we see that
  \[
  Z((e_b e_r)^n e_b, v) Z(\alpha_b, \alpha_r) Z(v, (e_b e_r)^n e_b) =
  \{\gamma_n\},
  \]
  and so the $\gamma_n$ all belong to $\intiso(G_\Lambda)$.

  The sets $Z((e_b e_r)^n e_b, (e_be_r)^n e_r)$ form a decreasing base
  of neighbourhoods of $\gamma := \big(z, (1,-1), z\big)$, and each
  $\gamma_n \in Z((e_b e_r)^n e_b, (e_be_r)^n e_r)$ Hence
  $\gamma_n \to \gamma$, giving
  $\gamma \in \overline{\intiso(G_\Lambda)}$.

  The elements
  \[
  \gamma'_n := \big((e_b e_r)^n e_b y, (1,-1), (e_be_r)^n e_r y\big)
  \in G_\Lambda
  \]
  satisfy $\gamma'_n \in Z((e_b e_r)^n e_b, (e_be_r)^n e_r)$ for each
  $n$, and so $\gamma'_n \to \gamma$. We claim that
  $\gamma'_n \not\in \Iso(G_\Lambda)$ for each $n$.  To see this, we
  calculate:
  \[
  \sigma^{((n+1),(n+1))}(r(\gamma'_n)) = \sigma^{(0,1)}(y) \not=
  \sigma^{(1,0)}(y) = \sigma^{((n+1),(n+1))}(s(\gamma'_n)).
  \]
  so $r(\gamma'_n) \not= s(\gamma'_n)$. Hence
  $\gamma = \lim \gamma'_n \not\in \intiso(G_\Lambda)$.
\end{example}

\begin{remark}
  The preceding example, combined with Corollary~\ref{cor:k-graph},
  shows that there exist $k$-graphs $\Lambda$ such that the subalgebra
  $M$ of $C^*(\Lambda)$ described above is maximal abelian and has the
  property that every representation of $C^*(\Lambda)$ that is
  faithful on $M$ is faithful, but is nevertheless not the range of a
  conditional expectation of $C^*(\Lambda)$.
\end{remark}

To finish, we clarify the relationship between Cartan subalgebras and
pseudo-diagonals.  On page~890 of \cite{nagrez:jlms12}, the authors
comment that the maximal abelian subalgebra that they construct in
each a graph algebra $C^*(E)$ is in fact a Cartan subalgebra. In their
subsequent paper \cite{nagrez:pams14}, they show that it is a
pseudo-diagonal. The relationship in general between these two
conditions is not addressed. We show that every Cartan subalgebra $B$
of a $C^*$-algebra $A$ is a pseudo-diagonal in $A$. This provides an
alternative proof of the assertion that $C^*(\intiso(G))$ is a
pseudo-diagonal in Corollary~\ref{cor:pseudodiagonal}, although the
proof via Cartan subalgebras provides less-direct information about
which pure states of $C^*(\intiso(G))$ have unique extension.

\begin{lemma}\label{lem:cartan->pd}
  Let $B$ be a Cartan subalgebra of a $C^*$-algebra $A$. Then $B$ is a
  pseudo-diagonal in $A$.
\end{lemma}

For the proof of Lemma~\ref{lem:cartan->pd}, we need to recall some
ideas from \cite{ren:imsb08}. A \emph{twist} over a Hausdorff \'etale
groupoid $G$ is a Hausdorff groupoid $\Sigma$ equipped with an
injective groupoid homomorphism $i : \T \times \go \to \Sigma$ and a
surjective groupoid homomorphism $q : \Sigma \to G$ such that the
kernel $\sset{\gamma \in \Sigma : q(\gamma) \in \go}$ of $q$ is the
image of $i$. We write $C_c(\Sigma, G)$ for the convolution algebra
\[
\sset{f \in C_c(\Sigma) : f(i(z, r(\gamma))\gamma) = zf(\gamma)\text{
    for all } \gamma \in \Sigma\text{ and } z \in \T}.
\]
There is an inclusion
$\iota : C_c(\go) \hookrightarrow C_c(G, \Sigma)$ such that each
$\iota(f)$ is supported on $i(\T \times \go)$ and satisfies
$\iota(f)(i(z,x)) = zf(x)$ for $z \in \T$ and $x \in \go$. We identify
$C_c(\go)$ with its image under $\iota$.

Also recall that a groupoid $G$ is \emph{topologically principal} if
the set of units in $\go$ with trivial isotropy is dense in $\go$.
That is, $\overline{\sset{x\in \go: G^x_x = \sset x}}=\go$. It is
worth pointing out that the condition we are here calling
topologically principal has gone under a variety of names in the
literature and that those names have not been used consistently (see
\cite{bcfs:sf14}*{Remark~2.3}).

\begin{proof}[Proof of Lemma~\ref{lem:cartan->pd}]
  We must show that $B$ is maximal abelian in $A$, that there is a
  faithful conditional expectation from $A$ onto $B$, and that the set
  of pure states of $B$ with unique extension to $A$ is weak$^*$-dense
  in the set of pure states of $B$.

  By \cite{ren:imsb08}*{Theorem~5.9(i)}, there exist a topologically
  principal \'etale groupoid $G$ and a twist $\Sigma$ over $G$ for
  which there exists an isomorphism $\pi : A \to C^*_r(G, \Sigma)$
  that carries $B$ to the canonical copy of $C_0(\go)$. So it suffices
  to show that there is a dense set of points $x$ in $\go$ for which
  the state $\widehat{x}(f) = f(x)$ on $C_0(\go)$ has unique extension
  to $C^*_r(G, \Sigma)$. Since $G$ is topologically principal, the set
  $\sset{x \in \go : G^x_x = \{x\}}$ is dense, so it suffices to show
  that if $G^x_x = \{x\}$, then $\widehat{x}$ has unique extension.

  The argument is very similar to that of
  Theorem~\ref{thm:uniqueness}(\ref{it:M state ext}), so we just give
  a quick sketch. By the argument preceding
  \cite{and:tams79}*{Theorem~3.2} we just have to show that for
  $a \in C_c(G, \Sigma)$ there exists a positive element
  $b \in C_0(\go)$ such that $\widehat{x}(b) = \|b\| = 1$ and
  $bab \in C_0(\go)$. Write $q : \Sigma \to G$ for the quotient
  map. Fix $a \in C_c(G, \Sigma)$. Use a partition of unity to express
  $a = \sum_{B \in F} a_B$ where $F$ is a finite collection of
  precompact open bisections of $G$ and each $a_B \in
  C_c(q^{-1}(B))$.
  For each $B$ such that $x \not\in B$, since $G^x_x = \{x\}$, there
  is a neighbourhood $V_B$ of $x$ such that $V_B B V_B =
  \emptyset$.
  And for $B$ such that $x \in B$, there is a neighbourhood $V_B$ of
  $x$ such that $V_B B V_B = V_B \subset \go$. Let
  $V := \bigcap_B V_B$ and choose $b \in C_c(V)$ such that $b(x) = 1$;
  that is $\widehat{x}(b) = 1$. Then
  $bab = \sum_{x \in B} b a_B b \in C_0(\go)$ by choice of the $V_B$.
\end{proof}


\begin{bibdiv}
  \begin{biblist}

    \bib{and:tams79}{article}{ author={Anderson, Joel},
      title={Extensions, restrictions, and representations of states
        on {$C^{\ast} $}-algebras}, date={1979}, ISSN={0002-9947},
      journal={Trans. Amer. Math. Soc.}, volume={249}, number={2},
      pages={303\ndash 329}, url={http://dx.doi.org/10.2307/1998793},
      review={\MR{525675 (80k:46069)}}, }

    \bib{bc:jfaa15}{article}{ author={B{\'e}dos, Erik}, author={Conti,
        Roberto}, title={Fourier series and twisted
        ${\rm C}^{\ast}$-crossed products}, journal={J. Fourier
        Anal. Appl.}, volume={21}, date={2015}, number={1},
      pages={32--75}, issn={1069-5869}, review={\MR{3302101}},
      doi={10.1007/s00041-014-9360-3}, }

    \bib{bla:operator06}{book}{ author={Blackadar, Bruce},
      title={Operator algebras}, series={Encyclopaedia of Mathematical
        Sciences}, publisher={Springer-Verlag}, address={Berlin},
      date={2006}, volume={122}, ISBN={978-3-540-28486-4;
        3-540-28486-9}, note={Theory of $C{^{*}}$-algebras and von
        Neumann algebras, Operator Algebras and Non-commutative
        Geometry, III}, review={\MR{2188261 (2006k:46082)}}, }

    \bib{bcfs:sf14}{article}{ author={Brown, Jonathan~H.},
      author={Clark, Lisa~Orloff}, author={Farthing, Cynthia},
      author={Sims, Aidan}, title={Simplicity of algebras associated
        to {\'e}tale groupoids}, journal={Semigroup Forum},
      volume={88}, year={2014}, number={2}, pages={433--452}, }

    \bib{bnr:jfa14}{article}{ author={Brown, Jonathan~H.},
      author={Nagy, Gabriel}, author={Reznikoff, Sarah}, title={A
        generalized {C}untz-{K}rieger uniqueness theorem for
        higher-rank graphs}, date={2014}, ISSN={0022-1236},
      journal={J. Funct. Anal.}, volume={266}, number={4},
      pages={2590\ndash 2609},
      url={http://dx.doi.org/10.1016/j.jfa.2013.08.020},
      review={\MR{3150172}}, }

    \bib{carkanshosim:jfa14}{article}{ author={Carlsen, Toke Meier},
      author={Kang, Sooran}, author={Shotwell, Jacob}, author={Sims,
        Aidan}, title={The primitive ideals of the Cuntz--Krieger
        algebra of a row-finite higher-rank graph with no sources},
      journal={J. Funct. Anal.}, volume={266}, date={2014},
      number={4}, pages={2570--2589}, issn={0022-1236},
      review={\MR{3150171}}, doi={10.1016/j.jfa.2013.08.029}, }

    \bib{cunkri:im80}{article}{ author={Cuntz, Joachim},
      author={Krieger, Wolfgang}, title={A class of
        $C^{\ast} $-algebras and topological Markov chains},
      journal={Invent. Math.}, volume={56}, date={1980}, number={3},
      pages={251--268}, issn={0020-9910}, review={\MR{561974
          (82f:46073a)}}, doi={10.1007/BF01390048}, }

    \bib{Dixmier}{book}{ author={Dixmier, Jacques},
      title={$C\sp*$-algebras}, note={Translated from the French by
        Francis Jellett; North-Holland Mathematical Library, Vol. 15},
      publisher={North-Holland Publishing Co., Amsterdam-New
        York-Oxford}, date={1977}, pages={xiii+492},
      isbn={0-7204-0762-1}, review={\MR{0458185 (56 \#16388)}}, }

    \bib{exe:bbms08}{article}{ author={Exel, Ruy}, title={Inverse
        semigroups and combinatorial $C^\ast$-algebras},
      journal={Bull. Braz. Math. Soc. (N.S.)}, volume={39},
      date={2008}, number={2}, pages={191--313}, issn={1678-7544},
      review={\MR{2419901 (2009b:46115)}},
      doi={10.1007/s00574-008-0080-7}, }

    \bib{huekumsim:jfa11}{article}{ author={an Huef, Astrid},
      author={Kumjian, Alex}, author={Sims, Aidan}, title={A
        Dixmier-Douady theorem for Fell algebras},
      journal={J. Funct. Anal.}, volume={260}, date={2011},
      number={5}, pages={1543--1581}, issn={0022-1236},
      review={\MR{2749438 (2012i:46066)}},
      doi={10.1016/j.jfa.2010.11.011}, }

    \bib{kel:general}{book}{ author={Kelley, John~L.}, title={General
        topology}, publisher={Van Nostrand}, address={New York},
      date={1955}, }

    \bib{kumpas:nyjm00}{article}{ author={Kumjian, Alex},
      author={Pask, David}, title={Higher rank graph
        {$C^\ast$}-algebras}, date={2000}, ISSN={1076-9803},
      journal={New York J. Math.}, volume={6}, pages={1\ndash 20},
      url={http://nyjm.albany.edu:8000/j/2000/6_1.html},
      review={\MR{1745529 (2001b:46102)}}, }

    \bib{kumpasraeren:jfa97}{article}{ author={Kumjian, Alex},
      author={Pask, David}, author={Raeburn, Iain}, author={Renault,
        Jean}, title={Graphs, groupoids, and Cuntz--Krieger algebras},
      journal={J. Funct. Anal.}, volume={144}, date={1997},
      number={2}, pages={505--541}, issn={0022-1236},
      review={\MR{1432596 (98g:46083)}}, doi={10.1006/jfan.1996.3001},
    }

    \bib{kumpasrae:pjm98}{article}{ AUTHOR = {Kumjian, Alex}, AUTHOR =
      {Pask, David}, AUTHOR = {Raeburn, Iain}, TITLE =
      {Cuntz-{K}rieger algebras of directed graphs}, JOURNAL =
      {Pacific J. Math.}, FJOURNAL = {Pacific Journal of Mathematics},
      VOLUME = {184}, YEAR = {1998}, NUMBER = {1}, PAGES = {161--174},
    }


    \bib{mrw:tams96}{article}{ author={Muhly, Paul~S.},
      author={Renault, Jean~N.}, author={Williams, Dana~P.},
      title={Continuous-trace groupoid {$C\sp \ast$}-algebras. {III}},
      date={1996}, ISSN={0002-9947},
      journal={Trans. Amer. Math. Soc.}, volume={348}, number={9},
      pages={3621\ndash 3641}, review={\MR{MR1348867 (96m:46125)}}, }

    \bib{nagrez:jlms12}{article}{ author={Nagy, Gabriel},
      author={Reznikoff, Sarah}, title={Abelian core of graph
        algebras}, date={2012}, ISSN={0024-6107},
      journal={J. Lond. Math. Soc. (2)}, volume={85}, number={3},
      pages={889\ndash 908},
      url={http://dx.doi.org/10.1112/jlms/jdr073},
      review={\MR{2927813}}, }

    \bib{nagrez:pams14}{article}{ author={Nagy, Gabriel},
      author={Reznikoff, Sarah}, title={Pseudo-diagonals and
        uniqueness theorems}, date={2014}, ISSN={0002-9939},
      journal={Proc. Amer. Math. Soc.}, volume={142}, number={1},
      pages={263\ndash 275},
      url={http://dx.doi.org/10.1090/S0002-9939-2013-11756-9},
      review={\MR{3119201}}, }

    \bib{Phil05}{article}{ AUTHOR = {Phillips, N. Christopher}, TITLE
      = {Crossed products of the {C}antor set by free minimal actions
        of {$\Bbb Z^d$}}, JOURNAL = {Comm. Math. Phys.}, FJOURNAL =
      {Communications in Mathematical Physics}, VOLUME = {256}, YEAR =
      {2005}, NUMBER = {1}, PAGES = {1--42}, ISSN = {0010-3616}, CODEN
      = {CMPHAY}, MRCLASS = {46L55 (37A55 37B50)}, MRNUMBER = {2134336
        (2006g:46107)}, MRREVIEWER = {Johannes Kellendonk}, DOI =
      {10.1007/s00220-004-1171-y}, URL =
      {http://dx.doi.org/10.1007/s00220-004-1171-y}, }
	
    \bib{ren:groupoid}{book}{ author={Renault, Jean}, title={A
        groupoid approach to {\cs}-algebras}, series={Lecture Notes in
        Mathematics}, publisher={Springer-Verlag}, address={New York},
      date={1980}, volume={793}, }

    \bib{ren:imsb08}{article}{ author={Renault, Jean}, title={Cartan
        subalgebras in {$C^*$}-algebras}, date={2008},
      ISSN={0791-5578}, journal={Irish Math. Soc. Bull.}, number={61},
      pages={29\ndash 63}, review={\MR{2460017 (2009k:46135)}}, }

    \bib{ren:xx13}{unpublished}{ author={Renault, Jean},
      title={Topological amenability is a {B}orel property},
      date={2013}, note={(arXiv:1302.0636 [math.OA])}, }

    \bib{sww:dm14}{article}{ author={Sims, Aidan}, author={Whitehead,
        Benjamin}, author={Whittaker, Michael~F.}, title={Twisted
        {\cs}-algebras associated to finitely aligned higher rank
        graphs}, date={2014}, journal={Documenta Math.}, number={19},
      pages={831\ndash 866}, }

    \bib{simwil:xx15}{article}{ author={Sims, Aidan},
      author={Williams, Dana~P.}, title={The primitive ideals of some
        \'etale groupoid $C^*$-algebras}, date={2015}, number={18},
      pages={1\ndash 20},
      journal={Algebras and Representation Theory}
      note={(arXiv:1501.02302 [math.OA])}, }


    \bib{szy:ijm02}{article}{ author={Szyma{\'n}ski, Wojciech},
      title={General Cuntz--Krieger uniqueness theorem},
      journal={Internat. J. Math.}, volume={13}, date={2002},
      number={5}, pages={549--555}, issn={0129-167X},
      review={\MR{1914564 (2003h:46083)}},
      doi={10.1142/S0129167X0200137X}, }

    \bib{wil:crossed}{book}{ author={Williams, Dana~P.},
      title={Crossed products of {$C{\sp \ast}$}-algebras},
      series={Mathematical Surveys and Monographs},
      publisher={American Mathematical Society}, address={Providence,
        RI}, date={2007}, volume={134}, ISBN={978-0-8218-4242-3;
        0-8218-4242-0}, review={\MR{MR2288954 (2007m:46003)}}, }

    \bib{yang:xx14}{unpublished}{ author={Yang, Dilian},
      title={Periodic higher rank graphs revisited}, date={2014},
      note={(arXiv:1403.6848 [math.OA])}, }

    \bib{yang:xxxx}{unpublished}{ author={Yang, Dilian},
      title={Cycline subalgebras are Cartan}, date={2014}, }

  \end{biblist}
\end{bibdiv}

\def\noopsort#1{}\def\cprime{$'$} \def\sp{^}

\end{document}